%% file: Limit_theorem_of_cycle_for_weight_permutations.tex
\documentclass[12pt,reqno]{amsart}
\usepackage{latexsym,a4,mathrsfs,amsthm,amsmath,amssymb}
\usepackage[latin1]{inputenc}
\usepackage{graphics}
\usepackage{fullpage}
\usepackage{epsfig}
\usepackage[english]{babel}
\usepackage{hyperref}
\usepackage{subfigure}
\newtheorem{theorem}{Theorem}[section]
\newtheorem{lemma}[theorem]{Lemma}
\newtheorem{definition}[theorem]{Definition}

\newtheorem{corollary}{Corollary}[theorem]
\theoremstyle{definition}

\numberwithin{equation}{section}
\newcommand{\R}{\mathbb{R}}
\newcommand{\C}{\mathbb{C}}

\newcommand{\Z}{\mathbb{Z}}
\newcommand{\N}{\mathbb{N}}
\newcommand{\eps}{\epsilon}
\newcommand{\la}{\lambda}

\newcommand{\Pb}[1]{\mathbb{P}\left[#1\right]}
\newcommand{\E}[1]{\mathbb{E}\left[#1\right]}

\newcommand{\ET}[1]{\mathbb{E}_{\Theta}\left[#1\right]}
\newcommand{\PT}[1]{\mathbb{P}_{\Theta}\left[#1\right]}
\newcommand{\EF}[1]{\mathbb{E}_{\mathrm{F}}\left[#1\right]}
\newcommand{\PF}[1]{\mathbb{P}_{\mathrm{F}}\left[#1\right]}
\newcommand{\slan}{\sum_{\la \vdash n}\frac{1}{z_\la}}
\newcommand{\sla}{\sum_{\la}\frac{1}{z_\la}}
\newcommand{\nth}[1]{[t^n]\left[ #1 \right]}

\newcommand{\set}[1]{\left\{#1\right\}}

\renewcommand{\Re}{\mathrm{Re}}

\begin{document}
\title[The generalized weighted probability measure]{The generalized weighted probability measure on the symmetric group and the asymptotic behavior of the cycles}
%

%
%

\author[A. Nikeghbali]{Ashkan Nikeghbali}
\address{Institut f\"ur Mathematik\\ Universit\"at Z\"urich\\ Winterthurerstrasse 190\\ 8057-Z\"urich,
Switzerland}
\email{ashkan.nikeghbali@math.uzh.ch}

\author[D. Zeindler]{Dirk Zeindler}
\address{Department of Mathematics\\University of York\\York\\YO10
5DD\\United Kingdom}
\email{dz549@york.ac.uk}

\maketitle

\begin{abstract}
The goal of this paper is to analyse the asymptotic behavior of the cycle process and the total number of cycles of weighted and generalized weighted random permutations  which are relevant models in physics and which extend the Ewens measure. We combine tools from combinatorics and complex analysis (e.g. singularity analysis of generating functions) to prove that under some analytic  conditions (on relevant generating functions) the  cycle process converges to a vector of independent Poisson variables and to establish a central limit theorem for the total number of cycles. Our methods allow us to obtain an asymptotic estimate of the characteristic functions of the different random vectors of interest together with an error estimate, thus having a control on the speed of convergence. In fact we are able to prove a finer convergence for the total number of cycles, namely \textit{mod-Poisson convergence}. From there we apply previous results on mod-Poisson convergence to obtain Poisson approximation for the total number of cycles as well as large deviations estimates. 
\end{abstract}

\tableofcontents
\newpage
%
\input{section_introduction}

\input{section_preparation}

\input{section_limit_theorem_cycle}

\input{section_limit_theorem_tot_cycle_numbers_hwang}

\input{section_examples}

\input{section_general_generating}
\bibliographystyle{acm}
\bibliography{literatur}

\end{document}

%% file: section_introduction.tex
\section{Introduction}
\label{sec_introduction}

In this paper we are interested in finding the asymptotic behaviour of the cycle structure and the number of cycles of  weighted random permutations which appear in mathematical biology and in theoretical physics. More precisely, we define the weighted and generalized weighted probability measures on the group of permutations $S_n$ of order $n$ as follows:

\begin{definition}
\label{def_weighted_probabililty_measure}
Let $\sigma\in S_n$ be given.
We write $C_j(\sigma)$ for the number of cycles of length $j$ in the decomposition of $\sigma$ as a product of disjoint cycles (see also Definition~\ref{def_Cm_and_K0n}). 

\begin{enumerate}
\item Let $\Theta = \left(\theta_m  \right)_{m\geq1}$ be given, with $\theta_j\geq0$ for every $j\geq 1$.
We define the generalized weighted measures as
%
\begin{align}
  \PT{\sigma}
  :=
  \frac{1}{h_n n!} \prod_{m=1}^n \theta_m^{C_m(\sigma)}
  \label{eq_PTheta_with_partition}
\end{align}
with $h_n = h_n(\Theta)$ a normalization constant and $h_0:=1$.
\item More generally, we define the generalized weighted measures as follows.   Let $F_m: \N \to \R_{>0}$  be given for $m\geq 1$ with $F_m(0)=1$. We then define
    \begin{align}
      \PF{\sigma}
      :=
      \frac{1}{n! h_n(F)} \prod_{m=1}^n F_m \bigl(C_m(\sigma) \bigr)
    \end{align}
    with $h_n(F)$ a normalization constant. It follows immediately from the definition that $\PT{.} = \PF{.}$ with $F_m(k) = \theta_m^k$.

\end{enumerate}

\end{definition}
When $\theta_j=1$ for all $j\geq 1$, the measure $\PT{.}$ is the uniform measure on permutations and these are well studied objects with  a long history (see e.g. the monograph \cite{barbour} for a complete account). In particular it was proven by Goncharov (\cite{Goncharov}) and Shepp and Loyd (\cite{Shepp}) that the process of cycle counts converges in distribution to a Poisson process, that is as $n\to\infty$,
\begin{equation}\label{blabla1}
\left( C_1^{(n)},C_2^{(n)},\ldots\right)\to^d (Z_1,Z_2,\ldots),
\end{equation}

where  the $Z_j$ are independent Poisson distributed random variables with $\mathbb{E}[Z_j]=\dfrac{1}{j}.$ There have also been further studies to analyze the rate of convergence for $$\mathbb{P}\left[ \left( C_1^{(n)},\ldots,C_b^{(n)}\right)\right]\to\mathbb{P}\left[ \left( Z_1,\ldots,Z_b\right)\right], n\to\infty$$where $b$ is a fixed integer. The above can be interpreted as a result on small cycles. There exist as well results  on large cycles, due to Kingman (\cite{Kingman}) and Vershik and Shmidt (\cite{Vershik}) who prove that the vector of renormalized and ordered length of the cycles converges in law to a Poisson-Dirichlet distribution with parameter $1$. Moreover if one notes $$K_{0n}=C_1^{(n)}+\ldots+C_n^{(n)}$$the number of cycles, then the distribution of $K_{0n}$ is well known and a central limit theorem can be proven:
\begin{equation}\label{blabla2}
\dfrac{K_{0n}-\log n}{\sqrt{\log n}}\to^d\mathcal{N}(0,1),
\end{equation}
where $\mathcal{N}(0,1)$ stands for a standard Gaussian random variable. In fact one can prove Poisson approximation results for $K_{0n}-1$ as well as large deviations estimates. For instance Hwang (\cite{Hwang95}) showed that for $k\sim x\log n$,
$$\mathbb{P}[K_{0n}=k]=\dfrac{(\log n)^{k-1}\exp(-\log n)}{(k-1)!}\left(\dfrac{1}{\Gamma(1+r)}+O\left(\dfrac{k}{\left(\log n\right)^2}\right) \right),$$where $r=(k-1)/\log n$.

Similar results exist if one considers the more general Ewens measure corresponding to $\theta_j=\theta>0$ for all $j$ in equation (\ref{eq_PTheta_with_partition}) defining $\PT{.}$. This measure was introduced by Ewens (\cite{Ewens}) to study population dynamics and has received much attention since. In particular (\ref{blabla1}) and (\ref{blabla2}) hold with $\mathbb{E}[Z_j]=\frac{\theta}{j}$ and $\log n$ replaced with $\theta \log n$ in the central limit theorem. Estimates on the rate of convergence as well as Poisson approximation results for $K_{0n}$ area available as well.

The measure $\PT{.} $ is thus  a natural extension of the Ewens measure and besides has a physical interpretation: indeed such a model appears in the study of the quantum Bose gas in statistical mechanics (see \cite{BU}, \cite{BVUD} and \cite{BUcmp}). There it is of interest to understand the structure of the cycles when the asymptotic behaviour of $\theta_j$ is fixed.  The case where $\theta_j\to\theta$, i.e. asymptotically the Ewens case, was also considered in \cite{random_permutations_with_cycle_weight}. Another random variable of interest that we shall not consider in this paper  is $L_1$, the length of the cycle that contains $1$, which can be interpreted as giving the length of a typical cycle, has also been considered in some of the above mentioned papers. It appeared in these works that obtaining the convergence in distribution of the cycle process or the central limit theorem for the number of cycles is a challenging problem. Indeed, there does not exist something such as the Feller coupling for the random permutations under the measure $\PT{.} $, since in general these measures do not possess any compatibility property between the different dimensions. The main important property of $\PT{.} $ is that it is invariant on conjugacy classes, and we shall exploit this fact.

In a recent skillful paper, Ercolani and Ueltschi (\cite{cycle_structure_ueltschi}) have obtained, under  a variety of assumptions on the asymptotic properties of the $\theta_j$'s, the convergence of the cycle process to a Poisson vector, as in (\ref{blabla1}), with this time $\mathbb{E}[Z_j]=\frac{\theta_j}{j}$. In some cases they obtained an equivalent for $\mathbb{E}[K_{0n}]$ and in some "degenerate cases" they proved that the total number of cycles converges in distribution to $1$ or to $1$ plus a Poisson random variable. But their method which is a subtle saddle point analysis of generating functions does not give any information on the asymptotic behaviour of the different random variables under consideration nor on the rate of convergence. But on the other hand because the method is general, they are able to cover concrete cases corresponding to a large variety of assumptions on the asymptotic behaviour of the $\theta_j$'s.

The goal of this paper is to bring a complementary point of view to the approach of  Ercolani and Ueltschi (\cite{cycle_structure_ueltschi}) and to provide sufficient conditions on the $\theta_j$'s, or more precisely on the analyticity properties of the generating series of $(\theta_m)_{m\geq1}$, under which one has (\ref{blabla1}), (\ref{blabla2}), estimates on the rate of convergence, as well as  Poisson approximation and large deviations estimates for the total number of cycles $K_{0n}$.

Our approach is based on the so called singularity analysis of the generating series of $(\theta_n)_{n\geq1}$ and is general enough to deal with the case of the more general measures $\PF{.}$. The starting point is the well known relation $$\sum_{n=1}^\infty h_n t^n=\exp \left( g(t)\right), \text{ with } g(t)=\sum_{n=1}^\infty \dfrac{\theta_n}{n} t^n $$which relates the generating series of the sequence $(h_n)_{n\geq1}$ to that of the sequence $(\theta_n)_{n\geq1}$, and its generalized version for the measure $\PF{.}$ (in fact these formulas will follow from some combinatorial lemmas which are useful to compute expectations or statistics-e.g. the characteristic function- under $\PF{.}$). To obtain an asymptotic for  $h_n$ as well as an estimate for the characteristic functions of the different random variables under consideration, it will reveal crucial to extract precise asymptotic  information with an error term on the coefficient of $[t^n][F(t,w)]$  for $F(t,w)=\exp(wg(t))S(t,w)$ where $S(t,w)$ is some holomorphic function in a domain containing $\{(t,w)\in \mathbb{C}^2; |t|\leq r, |w|\leq \hat{r}\}$ where $r$ is the radius of convergence of the generating series $g(t)$. This will reveal possible if one makes some assumptions on the analyticity property of the generating series $g(t)$ together with assumptions on the nature of its singularities at the point $r$ on the circle of convergence. Several results of this nature exist in the literature (see e.g. the monograph \cite{FlSe09}) but the relevant ones for us are due to Hwang (\cite{MR1724562} and \cite{hwang_thesis}).  Combining these results with some combinatorial lemmas, we are able to show that the cycle count process converges in law to a vector of Poisson process as in (\ref{blabla1}), with this time $\mathbb{E}[Z_{j}]=\frac{\theta_{j} r^{j}}{j}$, where we recall that $r$ is the radius of convergence of the generating series $g(t)$ of the sequence $(\theta_{n})_{{n\geq0}}$. We also have a an estimate for the speed of convergence. We are also able to prove a central limit theorem for $K_{{0n}}$, as well as Poisson approximation results and large deviations estimates. In fact our methods allow us to prove a stronger convergence result than the central limit theorem, namely \textit{mod-Poisson convergence}. This type of convergence  together with mod-Gaussian convergence was introduced in \cite{JKN} and further developed in \cite{Ashkan_mod_poisson}. It can be viewed as a higher order central limit theorem from which one can deduce many relevant information (in particular the central limit theorem).

More precisely the paper is organized as follows:
\begin{itemize}
\item In Section \ref{sec_preparation} we establish some basic combinatorial lemmas and review some facts from complex analysis (e.g. singularity analysis theorems and the Lindel\"of integral representation theorem) from which we deduce the asymptotic behaviour of the normalization constant $h_{n}$ (under some assumptions on the generating function $g(t)$).
\item In Section \ref{sec_central_limit_number_cycles} we prove the Poisson convergence for the cycle process together with a rate of convergence;
\item Section \ref{sec_central_limit_tot_number_cycles} is devoted to various limit theorems for the total number of cycles $K_{0n}$;
\item Section \ref{sec_examples} contains some examples;
\item In Section \ref{sec_more_general_measures}  we prove general limit theorems under the more general measure $\PF{.}$. We shall illustrate these results with the example of exp-polynomial weights and a toy example of spatial random  permutations which plays an important role in physics (see e.g. \cite{spatial_random_permutations}). Here we consider the simpler case where the lattice is fixed. It is our hope that our methods can be adapted to deal with more complicated cases where the lattice is not fixed anymore.
\end{itemize}

%% file: section_preparation.tex
\section{Combinatorics and singularity analysis}
\label{sec_preparation}
\subsection{Combinatorics of $S_{n}$ and generating functions}\label{sec_conj_class_of_Sn}

We recall  in this section some basic facts about $S_n$ and partitions, and at the end of the section state a useful lemma to perform averages on the symmetric group. We only give here  a very short overview and refer to \cite{barbour} and \cite{macdonald} for more details.

We first analyse the conjugation classes of $S_n$ since
all probability measures and functions considered in this paper are invariant under conjugation. It is well known that the conjugation classes of $S_n$ can be parametrized with partitions of $n$.
\begin{definition}
 \label{def_part}
A \emph{partition} $\la$ is a sequence of non-negative integers $\la_1 \ge \la_2 \ge \cdots$ eventually trailing to 0's, which we usually omit. We use the notation $\la=(\la_1,\la_2,\cdots,\la_l)$. \\
The \emph{length} $l(\la)$ of $\la$ is the largest $l$ such that $\la_l \ne 0$. We define the \emph{size} $|\lambda|:= \sum_m \lambda_m$. We call $\la$ a partition of $n$ if  $|\la| = n$.
We use the notation
\begin{align*}
\sum_{\la \vdash n} (..)
:=
\sum_{\la \text{ partition of }n } (..)
\text{ and }
 \sum_{\la}  (..)
:=
\sum_{\la \text{ partition} } (..).
\end{align*}
\end{definition}
Let $\sigma\in S_n$ be arbitrary. We can write $\sigma = \sigma_1\cdots \sigma_l$ with $\sigma_i$ disjoint cycles of length $\la_i$.
Since disjoint cycles commute, we can assume that $\la_1\geq\la_2\geq\cdots \geq\la_l$.
We call the partition $\la=(\la_1,\la_2,\cdots,\la_l)$ the \emph{cycle-type} of $\sigma$. We write $\mathcal{C}_\la$ for the set of all $\sigma \in S_n$ with cycle type $\la$.
One  can show that two elements $\sigma,\tau\in S_n$ are conjugate if and only if
$\sigma$ and $\tau$ have the same cycle-type and that the $C_\la$ are the conjugacy classes of $S_n$ (see e.g. \cite{macdonald} for more details).

\begin{definition}
\label{def_Cm_and_K0n}
Let $\sigma\in S_n$ be given with cycle-type $\la$.
The cycle numbers $C_m$ and the total number of cycles $K_{0n}$ are defined as
\begin{align}
  C_m = C_m^{(n)}(\sigma) := \#\set{i ;\la_i = m} \text{ and }   K_{0n}:= \sum_{m=1}^n C_m^{(n)}.
\end{align}
\end{definition}
The functions $C_m$ and $K_{0n}$ depend only on the cycle type and are thus class functions (i.e. they are constant on conjugacy classes).

All expectations in this paper have the form $\frac{1}{n!} \sum_{\sigma\in S_n} u(\sigma)$ for a certain class function $u$.
Since $u$ is constant on conjugacy classes, it is more natural to sum over all conjugacy classes.
We thus need to know the size of each conjugacy class.

\begin{lemma}
\label{lem_size_of_conj_classes}
We have
\begin{align}
|\mathcal{C}_\la|=\frac{|S_n|}{z_\la}
\text{ with }
z_\la:=\prod_{m=1}^{n} m^{c_m}c_m!
\text{ and }c_m = c_m(\la) = \# \set{\la_i;\la_i = m},
\end{align}
and
\begin{align}
  \frac{1}{n!} \sum_{\sigma\in S_n} u(\sigma)
  =
  \sla u(\mathcal{C}_\la)
\end{align}
for a class function $u: S_n \to \C$.
\end{lemma}
\begin{proof}
 The first part can be found in \cite{macdonald} or in \cite[chapter 39]{bump}. The second part follows immediately from the first part.
\end{proof}

Given a sequence $(g_{n})_{n\geq1}$ of numbers, one can encode the information about this sequence into a formal power series called the generating series.
\begin{definition}
\label{def_gneranting_function}
Let $\bigl(g_n \bigr)_{n\in\N}$ be a sequence of complex numbers. We then define
the (ordinary) generating function of $\bigl(g_n \bigr)_{n\in\N}$ as the formal power series
    \begin{align}
      G(t) = G(g_n,t) = \sum_{n=0}^\infty g_n t^n.
    \end{align}
%
\end{definition}
\begin{definition}
  Let $G(t) =\sum_{n=0}^\infty g_n t^n$ be a formal power series. We then define
$\nth{G} := g_n$, i.e. the coefficient of $t^n$ in $G(t)$
\end{definition}
The reason why generating functions are useful is that it is often possible to write down a generating function
without knowing $g_n$ explicitly. Then one can try to use tools from analysis to extract information about $g_{n}$, for large $n$, from the generating function. It should be noted that
there are several variants in the definition of generating series and we shall use several of them and still call all of them generating series without risk of confusion.
We will also later replace  (see Section~\ref{sec_central_limit_tot_number_cycles}) $g_n$ by holomorphic functions $g_n(w)$.
Such generating functions are then called bivariate generating functions. Again for simplicity we shall still call them generating functions.

We now introduce two generating functions (in the broad sense) which will play a crucial role in our study of random weighted permutations under the measure   $\PT{(.)}$. For $\Theta=(\theta_{n})_{n\geq1}$, we set
\begin{align}
  g_\Theta(t)
  &:=
  \sum_{k=1}^\infty \frac{\theta_k}{k} t^k
  \text{ and }
  G_\Theta(t)
  :=
  \exp\left( \sum_{k=1}^\infty \frac{\theta_k}{k} t^k \right)
  \label{eq_def_gen_hn_with_theta}
\end{align}
For now, $g_\Theta(t)$ and $G_\Theta(t)$ are just formal power series.
We will see in Section~\ref{sec_central_limit_number_cycles} and in Section~\ref{sec_central_limit_tot_number_cycles} that the asymptotic behaviour of $C^{(n)}_m$ and $K_{0n}$ depend on the analytic properties of $g_\Theta(t)$, essentially because of the remarkable well known  identity (that we will quickly derive below)
%

%
%
\begin{align}
  G_\Theta(t)
=
\sum_{n=0}^\infty h_n t^n.
\label{eq_def_gen_Gt_with_hn}
\end{align}

One of the main tools in this paper to compute generating series is the following lemma (or cycle index theorem) of which we shall prove a more general form in Section \ref{sec_more_general_measures} to deal with the more general measure $\PF{.}$.

\begin{lemma}
\label{lem_cycle_index_theorem}
Let $(a_m)_{m\in\N}$ be a sequence of complex numbers. Then
\begin{align}
\label{eq_symm fkt}
\sla \left(\prod_{m=1}^{l(\la)} a_{\la_{m}}\right) t^{|\la|}
=
\sla \left(\prod_{m=1}^{\infty} (a_m t^m)^{C_m} \right)
=
\exp\left(\sum_{m=1}^{\infty}\frac{1}{m} a_m t^m\right)
\end{align}
with the same $z_\la$ as in Lemma~\ref{lem_size_of_conj_classes}.\\
If one of the sums in \eqref{eq_symm fkt} is absolutely convergent then so are the others.
\end{lemma}
\begin{proof}
The first equality follows immediately from the definition of $C_m$.
The proof of the second equality in \eqref{eq_symm fkt} can be found in \cite{macdonald} or can be directly verified using the definitions of $z_\la$ and the exponential function.
The last statement follows from dominated convergence.
\end{proof}
We now use this lemma to show \eqref{eq_def_gen_Gt_with_hn}, i.e. $G_\Theta(t) = \sum_{n=0}^\infty h_n t^n$.
We know from \eqref{eq_PTheta_with_partition} that
\begin{align}
h_n
=
\frac{1}{n!} \sum_{\sigma\in S_n} \prod_{m=1}^{n} \theta_{m}^{C_m}
=
\frac{1}{n!} \sum_{\la \vdash n} \frac{n!}{z_\la} \prod_{m=1}^{l(\la)} \theta_{\la_m}
=
\sum_{\la \vdash n} \frac{1}{z_\la} \prod_{m=1}^{l(\la)} \theta_{\la_m}.
\end{align}
It now follows from lemma~\ref{lem_cycle_index_theorem} that
\begin{align}
  \sum_{n=1}^\infty h_n t^n
  &=
  \sum_{\la} \frac{1}{z_\la}  t^{|\la|} \prod_{m=1}^{l(\la)} \theta_{\la_m}
  =
  \exp \left( \sum_{m=1}^\infty \frac{\theta_m}{m} t^m \right)
  =
  G_\Theta(t)
\end{align}
This proves \eqref{eq_def_gen_Gt_with_hn}.\\
\subsection{Singularity analysis}

If a generating function $g(t)$ is given then a natural question is:
what is $\nth{g}$ and what is the asymptotic behaviour of $\nth{g}$. If $g(t)$ is holomorphic near $0$ then one can use
Cauchy's integral formula to do this.
%
Unfortunately it is often difficult to compute the integral explicitly,
but there exist several other results  which allow to achieve this task.
One such theorem is due to Hwang \cite{hwang_thesis}, and we prepare it with some preliminary definition and notation.

\begin{definition}
\label{def_delta_0}
Let $0< r < R$ and $0 < \phi <\frac{\pi}{2}$ be given. We then define
\begin{align}
\Delta_0 = \Delta_0(r,R,\phi) = \set{ z\in \C ; |z|<R, z \neq r ,|\arg(z-r)|>\phi}
\label{eq_def_delta_0}
\end{align}
\begin{figure}[ht!]
 \centering
 \includegraphics[width=.4\textwidth]{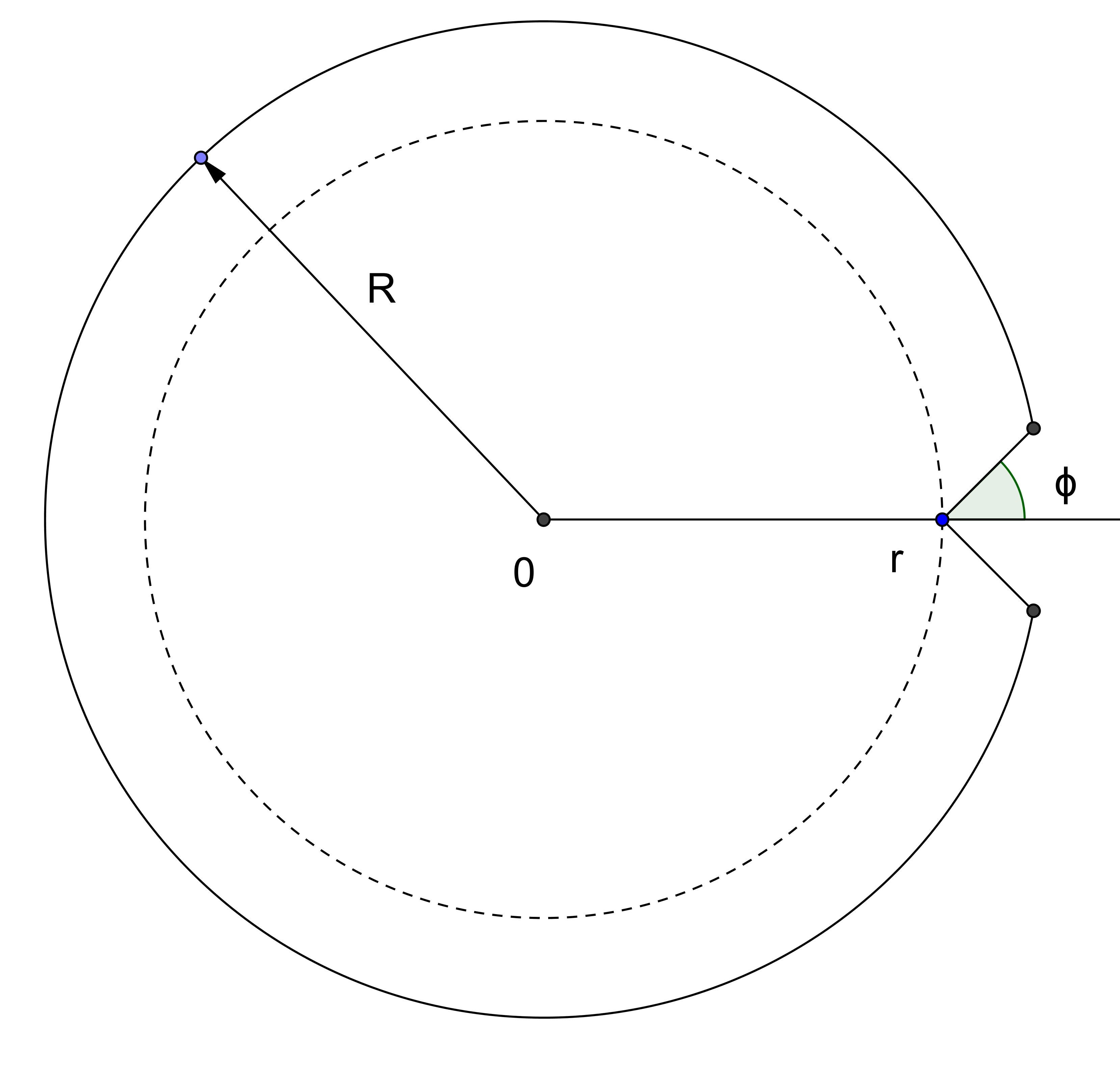}
 \caption{Illustration of $\Delta_0$}
 \label{fig_delta_0}
\end{figure}
\end{definition}

\begin{definition}
 \label{def_function_class_F_alpha_r}
Let $g(t)$ and $\theta \geq 0, r>0$ be given. We then call $g(t)$ of class $\mathcal{F}(r,\theta)$ if
\begin{enumerate}
 \item there exists $R>r$ and $0 < \phi <\frac{\pi}{2}$ such that $g(t)$ is holomorphic in $\Delta_0(r,R,\phi)$,
 \item there exists a constant $K$ such that
   \begin{align}
   g(t) = \theta \log\left( \frac{1}{1-t/r} \right) + K + O \left( t-r \right)   \text{ for } t\to r.
   \label{eq_class_F_r_alpha_near_r}
   \end{align}
\end{enumerate}
We shall use the shorter notation $g(t) \in \mathcal{F}(r,\theta)$.
\end{definition}
We emphasize only the dependence on $\theta$ and $r$ since the other constants do not appear in the main results.
We now state an important theorem due to Hwang:
\begin{theorem}[Hwang \cite{hwang_thesis}]
\label{thm_hwang}
Let $F(t,w) = e^{w g(t)} S(t,w)$ be given. Suppose that
\begin{enumerate}
 \item Let $g(t)$ is of class $\mathcal{F}(r,\theta)$,
 \item \label{item_S_holo} $S(t,w)$ is holomorphic in $(t,w)$ for $|t| \leq r$ and  $|w| \leq \hat{r}$ for some $\hat{r}>0$, i.e. $S(t,w)$ is holomorphic in $(t,w)$ in a domain containing the set $\set{ (t,w)\in \C^2;  |t| \leq r, |w| \leq \hat{r} }$.
\end{enumerate}
Then
\begin{align}
  \nth{F(t,w)}
=
\frac{e^{Kw} n^{w\theta-1} }{r^n} \left(\frac{S(r,w)}{\Gamma(\theta w)} + O\left(  \frac{1 }{n}\right)\right)
\label{eq_thm_hwang}
\end{align}
uniformly for $|w| \leq \hat{r} $ and with the same $K$ as in \eqref{eq_class_F_r_alpha_near_r}.
\end{theorem}
\begin{proof}
The idea of the proof is to take a suitable Hankel contour and to estimate the integral over each piece.
The details can be found in \cite[chapter 5]{hwang_thesis}.
\end{proof}
\textbf{Remark:}

We use most times $w = 1$ and $S$ independent of $w$. One can also compute lower order error terms if one has more terms in the expansion of $g(t)$ near $r$. \\

A natural question at this point is: how can one prove that $g(t)$ is of class $\mathcal{F}(r,\theta)$?
It is most times easy to compute the radius of convergence of $g(t)$, but it is not obvious how to show that $g(t)$ is holomorphic in some $\Delta_0$. A way to achieve this is through Lindel\"of's integral represention:

\begin{theorem}[Lindel\"of's integral represention]
\label{thm_lindelof_integral_representation}
Let $\phi(z)$ be a holomorphic function for $\Re(z)>0$, satisfying
\begin{align}
 |\phi(z)| < C e^{A|z|} \text{ for } |z| \to \infty \text{ and }\Re(z)\geq \frac{1}{2} \text{ with some } A\in ]0,\pi[, C > 0 .
\end{align}
Let  $g(t):= \sum_{k=1}^\infty \phi(k) (-t)^k$. The radius of convergence of $g(t)$ is $e^{-A}$ and
\begin{align}
g(t)
=
\frac{-1}{2\pi i}
\int_{1/2 - i\infty}^{1/2 + i\infty}
\phi(z) t^z \frac{\pi}{\sin(\pi z)} \ dz.
\label{eq_lindelof_integral_representation}
\end{align}
Furthermore $g(t)$ can be holomorphically continued to the sector $-(\pi-A) < \arg(t) < (\pi-A)$.
\end{theorem}
\begin{proof}
See \cite[Theorem~2]{MR2578896}.
\end{proof}
In many situations Theorem~\ref{thm_lindelof_integral_representation} allows us  to prove holomorphicity in a domain $\Delta_0$, but does not give any information about the behaviour of $g(t)$ near the singularity.
One way to compute the asymptotic behaviour of $g(t)$ near the singularity is to use the Mellin transform.
We do not introduce here the Mellin transform since this would take  us to far away from the topic of this paper.
We would rather refer to \cite{MR1337752} for an introduction and to \cite[Section~3]{MR1678788} for an application to the polylogarithm.\\

Theorem~\ref{thm_lindelof_integral_representation} is not always so easy to apply and the computation of the asymptotic behaviour near the singularity is often very difficult.
An alternative approach is to combine singularity analysis with more elementary methods.
The idea is to write $F = F_1 F_2$ in a way that we can apply singularity analysis to $F_1$ and can estimate the growth rate of $\nth{F_2}$. One  can then compute the coefficient $\nth{F}$ directly and  conclude with elementary analysis. This method is called the \textit{convolution method}.
We now introduce a class of functions which is suitable for this approach.
\begin{definition}
 \label{def_function_class_eF_alpha_r_gamma}
Let $g(t)$ and $\theta\geq 0, r>0, 0<\gamma \leq 1$ be given. We then call $g(t)$ of class $e\mathcal{F}(r,\theta,\gamma)$ if there exists an analytic function $g_0$ in the disc $\{|z|<r\}$ such that
\begin{align}
   g(t) = \theta \log\left( \frac{1}{1-t/r} \right) + g_0(t)   \text{ for } t\to r \text{ with }\nth{g_0} = O(r^{-n} n^{-1-\gamma}).
   \label{eq_class_eF_alpha_r_gamma_near_r}
\end{align}
\end{definition}
We can now state a singularity analysis theorem corresponding to the class $e\mathcal{F}(r,\theta,\gamma)$:
\begin{theorem}[Hwang \cite{MR1724562}]
\label{thm_hwang_2}
Let $F(t,w) = e^{w g(t)} S(t,w)$ be given such that
\begin{enumerate}
 \item  $g(t)$ is of class $e\mathcal{F}(r,\theta,\gamma)$,
 \item   $S(t,w)$ is holomorphic in $(t,w)$ for $|t| \leq r$ and  $|w| \leq \hat{r}$ for some $\hat{r}>0$.
\end{enumerate}
We then have
\begin{align}
\nth{F(t,w)}
=
\frac{e^{w K} n^{w\theta-1} }{r^n} \frac{S(r,w)}{\Gamma(\theta w)} + R_n(w)
\end{align}
with $K = g_0(r)$ and
\begin{align}
  R_n(w)
  =
  \left\{
    \begin{array}{ll}
      O \left(\frac{n^{\theta\Re(w) - 1- \gamma} \log(n)}{r^n} \right), & \hbox{ if }\Re(w) \geq 0 \\
      O \left(\frac{ n^{-1 - \gamma} }{r^n} \right), &  \hbox{ if } \Re(w) < 0.
    \end{array}
  \right.
\end{align}
uniformly for bounded $w$.
\end{theorem}
\begin{proof}
See \cite{MR1724562}.
\end{proof}
We can now compute the asymptotic behaviour of $h_n$.
\begin{lemma}
\label{lem_asymptotic_hn}
Let $g_\Theta(t)$ be of class $\mathcal{F}(r,\theta)$ or of class $e\mathcal{F}(r,\theta,\gamma)$. We then have
\begin{align}
\frac{1}{h_n}
\sim
r^n \frac{\Gamma(\theta)}{e^{K} n^{\theta-1}},\;n\to\infty.
\label{eq_asymptotic_hn}
\end{align}
with the same $K$ as in \eqref{eq_class_F_r_alpha_near_r} resp. as in Theorem~\ref{thm_hwang_2}.
\end{lemma}

\begin{proof}
We have proven that $\sum_{n=0}^\infty h_n t^n = G_\Theta(t) = \exp( g_\Theta(t))$.
We thus can apply Theorem~\ref{thm_hwang} resp. Theorem~\ref{thm_hwang_2} with $g(t) = g_\Theta(t)$ for $w=1$ and $S(t,w) \equiv 1$. We get
\begin{align}
  h_n
\sim
\frac{e^K n^{\theta-1}}{r^n \Gamma(\theta)},\;n\to\infty
\end{align}
\end{proof}

\textbf{Remark:}

There exist several other versions of the Theorems~\ref{thm_hwang} and Theorems~\ref{thm_hwang_2}.
For example one can replace $\log(1-t/r)$ by other functions (see \cite{Flajolet_generating}) or allow more than one singularity (see \cite[chapter VI.5]{FlSe09}.

%% file: section_limit_theorem_cycle.tex
\section{Limit theorem for the cycle numbers}
\label{sec_central_limit_number_cycles}

In this section we establish the convergence in distribution of the cycle process to a vector of Poisson random variables.

\begin{theorem}
\label{thm_cycle_numbers_generating}
 Let $b\in\N$ be fixed. We then have as formal power series
\begin{align}
  \sum_{n=0}^\infty
   h_n \ET{\exp \left(i\sum_{m=1}^b s_m C_m^{(n)} \right)} t^n
    &=
\exp\left( \sum_{m=1}^b \frac{\theta_m}{m} (e^{is_m}-1) t^m \right)  G_\Theta(t).
\label{eq_generating_cycles}
\end{align}
If $g_\theta$ is of class $\mathcal{F}(\theta,r)$, then
\begin{align}
 \ET{\exp \left(i\sum_{m=1}^b s_m C_m^{(n)} \right)}
=
\exp\left( \sum_{m=1}^b \frac{\theta_m}{m} (e^{is_m}-1) r^m \right) + O\left(\frac{1}{n} \right).
\label{eq_asymptotic_cycles_F_r_theta}
\end{align}
If $g_\Theta(t)$ is of class $e\mathcal{F}(r,\theta,\gamma)$, then
\begin{align}
\ET{\exp \left(i\sum_{m=1}^b s_m C_m^{(n)} \right)}
=
\exp\left( \sum_{m=1}^b \frac{\theta_m}{m} (e^{is_m}-1) r^m \right) + O\left( \frac{\log(n)}{n^\gamma} \right)
\label{eq_asymptotic_cycles_eF_r_theta_gamma}
\end{align}
\end{theorem}
The convergence result now  follows immediately from Theorem~\ref{thm_cycle_numbers_generating}:
\begin{corollary}
\label{cor_limit_theorem_cycles}
Let $\Theta = (\theta_m)_{m\in\N}$ be given and $S_n$ be endowed with $\PT{(.)}$.
Assume that $g_\Theta(t)$ is of class $\mathcal{F}(r,\theta)$ or of class $e\mathcal{F}(r,\theta, \gamma)$.
We then have for each $b\in\N$
\begin{align}
  \left(C_1^{(n)}, C_2^{(n)},  \cdots C_b^{(n)} \right)
  \stackrel{d}{\to}
  \left(Y_1,\cdots,Y_b\right)
\end{align}
with $Y_1,\cdots,Y_b$ independent Poisson distributed random variables with $\E{Y_m} = \frac{\theta_m}{m}r^m$.\\
Let $c_1,\cdots,c_n$ be integers. We then have for $g_\Theta \in \mathcal{F}(r,\theta)$
\begin{align*}
 \left|\PT{C_1^{(n)}=c_1, \cdots ,C_b^{(n)}=c_b} - \PT{Y_1=c_1, \cdots ,Y_b=c_b} \right|
=
O\left(\frac{1}{n} \right)
\end{align*}
and for $g_\Theta \in e\mathcal{F}(r,\theta,\gamma)$
\begin{align*}
 \left|\PT{C_1^{(n)}=c_1, \cdots ,C_b^{(n)}=c_b} - \PT{Y_1=c_1, \cdots ,Y_b=c_b} \right|
=
O\left( \frac{\log(n)}{n^\gamma} \right)
\end{align*}
\end{corollary}
\begin{proof}
The first part follows immediately from L\'evy's continuity theorem. 
To prove the second part, we use the Fourier inversion formula. 
Let $\psi(s_1,\cdots,s_b)$ be the characteristic function of $ C_1^{(n)}, C_2^{(n)},  \cdots C_b^{(n)}$ and 
$\phi(s_1,\cdots,s_b)$ the characteristic function of $Y_1,\cdots,Y_b$. 
Equation \eqref{eq_asymptotic_cycles_F_r_theta} resp. \eqref{eq_asymptotic_cycles_eF_r_theta_gamma} shows that
$|\psi-\phi| = O\left(\frac{1}{n} \right)$ resp. $|\psi-\phi| = O\left( \frac{\log(n)}{n^\gamma} \right)$ with $O(.)$ independent of $s_1,\cdots,s_d$. On the other hand we have
\begin{eqnarray*}
\PT{C_1^{(n)}=c_1, \cdots ,C_b^{(n)}=c_b} - \PT{Y_1=c_1, \cdots ,Y_b=c_b} \\
=
\frac{1}{(2\pi)^b} \int_{[-\pi,\pi ]^b} \bigl( \psi(s_1,\cdots,s_b) -\phi(s_1,\cdots,s_b)\bigr) e^{-\sum_{m=1}^b i c_m s_m} \ ds_1 \cdots ds_b.
\end{eqnarray*}
This proves the corollary.

\end{proof}
\begin{proof}[Proof of Theorem~\ref{thm_cycle_numbers_generating}]
We first compute the generating function in \eqref{eq_generating_cycles}.
%
The factor $h_n$ in \eqref{eq_generating_cycles} is necessary to use Lemma~\ref{lem_cycle_index_theorem}.
We have
\begin{align}
  h_n \ET{\exp \left(i\sum_{m=1}^b s_m C_m^{(n)} \right)}
  &=
  h_n \frac{1}{h_n}\slan  {\exp \left(i\sum_{m=1}^b s_m C_m^{(n)}\right) }  \prod_{m=1}^{l(\la)} \theta_{\la_m} \nonumber\\
  &=
  \slan
  \left(\prod_{m=1}^b (\theta_m e^{is_m})^{C_m^{(n)}} \right)
  \left(\prod_{m=b+1}^{l(\la)} (\theta_m)^{C_m^{(n)}} \right)
\end{align}
We now apply Lemma~\ref{lem_cycle_index_theorem} with
$ a_m = \left\{
        \begin{array}{ll}
         \theta_m e^{is_m}, & \hbox{ if }1\leq m \leq b    , \\
         \theta_m, & \hbox{if } b < m.
        \end{array}
        \right.$
We get
\begin{align}
  \sum_{n=0}^\infty
   h_n \ET{\exp \left(i\sum_{m=1}^b s_m C_m^{(n)} \right)} t^n
  &=
  \sla  t^{|\la|} \left(\prod_{m=1}^b (\theta_m e^{is_m})^{C_m^{(n)}} \right)
  \left(\prod_{m=b+1}^{l(\la)} (\theta_m)^{C_m^{(n)}} \right) \nonumber\\
  &=
  \exp\left( \sum_{m=1}^b \frac{\theta_m e^{is_m}}{m} t^m + \sum_{m=b+1}^\infty \frac{\theta_m}{m} t^m \right) \nonumber\\
  &=
\exp\left( \sum_{m=1}^b \frac{\theta_m}{m} (e^{is_m}-1) t^m \right)  G_\Theta(t).
\end{align}
This proves \eqref{eq_generating_cycles}.
The proof of \eqref{eq_asymptotic_cycles_F_r_theta} is very similar to the proof of \eqref{eq_asymptotic_cycles_eF_r_theta_gamma}. The only difference is that one has to apply Theorem~\ref{thm_hwang} if $g_\Theta(t) \in \mathcal{F}(r,\theta)$ and Theorem~\ref{thm_hwang_2} if $g_\Theta(t) \in e\mathcal{F}(r,\theta,\gamma)$.
We thus only prove  \eqref{eq_asymptotic_cycles_F_r_theta}.\\
The function $\sum_{m=1}^b \frac{\theta_m}{m} (e^{is_m}-1) t^m$ is a polynomial in $t$ and is therefore holomorphic on the whole complex plane.
We  can thus apply Theorem~\ref{thm_hwang} with 
\begin{align*}
  S(t) = S(t,w) = \exp\left( \sum_{m=1}^b \frac{\theta_m}{m} (e^{is_m}-1) t^m \right)
\end{align*}
to obtain:
%
%
\begin{align}
 h_n \ET{\exp \left(i\sum_{m=1}^b s_m C_m^{(n)} \right)}
=
 S(r) \frac{e^{K} n^{\theta-1} }{r^{n}\Gamma(\theta)}
+ O\left(  \frac{n^{\theta-2}}{r^n}   \right).
\label{eq_asymptotic_char_fkt_mal_hn}
\end{align}
For $g_\Theta(t) \in e\mathcal{F}(r,\theta,\gamma)$ we get a similar expression.
We have computed in Lemma~\ref{lem_asymptotic_hn} that
\begin{align*}
\frac{1}{h_n}
\sim
r^n \frac{\Gamma(\theta)}{e^{K} n^{\theta-1}}.
\end{align*}
This together with \eqref{eq_asymptotic_char_fkt_mal_hn} proves the theorem for $g_\Theta(t) \in \mathcal{F}(r,\theta)$.
The argumentation for $g_\Theta(t) \in e\mathcal{F}(r,\theta,\gamma)$ is similar and we omit it.
%
\end{proof}

%% file: section_limit_theorem_tot_cycle_numbers_hwang.tex
\section{The total number of cycles}
\label{sec_central_limit_tot_number_cycles}

We prove in this section a central limit theorem and mod-Poisson convergence for $K_{0n}$. From the mod-Poisson convergence we deduce Poisson approximation results for $K_{0n}$ as well as large deviations estimates.
As before we use generating functions.

\begin{lemma}
 \label{lem_generating_K0n}
We have for each  $w\in\C$ as formal power series
\begin{align}
\sum_{n=0}^\infty h_n \ET{\exp\bigl( w K_{0n}  \bigr)} t^n
=
\sum_{n=0}^\infty h_n \ET{\exp\left( w \sum_{m=1}^n C_m  \right)} t^n
=
\exp\bigl( e^w g_\Theta (t) \bigr).
\end{align}
If $g_\Theta(t)$ is of class $\mathcal{F}(r,\theta)$, then
\begin{align}
\ET{\exp\bigl( is K_{0n}  \bigr)}
 =
n^{\theta(e^{is} -1)} e^{K \left(e^{is} -1 \right)} \left( \frac{\Gamma(\theta)}{\Gamma(\theta e^{is})} + O\left(\frac{1}{n} \right) \right)
\label{eq_mod_poisson_class_F_r_theta_plus_error}
\end{align}
with $O(.)$ uniform for bounded $w$.
If $g_\Theta(t)$ is of class $e\mathcal{F}(r,\theta,\gamma)$, then
\begin{align}
\ET{\exp\bigl( is K_{0n}  \bigr)}
 =
n^{\theta(e^{is} -1)}e^{K \left(e^{is} -1 \right)} \frac{\Gamma(\theta)}{\Gamma(\theta e^{is})} + O\left( \frac{\log(n)}{n^\gamma} \right)
\label{eq_mod_poisson_class_eF_r_theta_gamma_plus_error}
\end{align}
with $O(.)$ uniform for bounded $w$.
\end{lemma}

\begin{proof}
To prove the first part, one can use exactly the same argumentation as in \eqref{eq_generating_cycles}.
We thus omit the details.
To prove the second part, we use Theorem~\ref{thm_hwang} for $g_\Theta(t) \in \mathcal{F}(r,\theta)$ to obtain
\begin{align}
\nth{\exp\bigl( e^w g_\Theta (t) \bigr)}
=
\frac{e^{Ke^{w}} n^{e^w \theta-1} }{r^n } \left( \frac{1}{\Gamma(\theta e^w)} +O\left(\frac{1}{n} \right) \right)
%
\end{align}
with $O(.)$ uniform for bounded $w$. We thus get with Lemma~\ref{lem_asymptotic_hn}
\begin{align}
 \ET{\exp\bigl( is K_{0n}  \bigr)}
 =
n^{e^{is} \theta-\theta}e^{K \left(e^{is} -1 \right)} \left(\frac{\Gamma(\theta)}{\Gamma(\theta e^{is})} +O\left(\frac{1}{n} \right)\right).
\label{eq_asyptotic_of_chara_Kon}
\end{align}
The argumentation for $g_\Theta(t) \in e\mathcal{F}(r,\theta,\gamma)$ is similar.
\end{proof}
We  can now prove a central limit theorem for $K_{0n}$.

\begin{theorem}
\label{thm_limit_theorem_total_cycles}
Let $g_\Theta(t)$ be of class $\mathcal{F}(r,\theta)$ or of class $e\mathcal{F}(r,\theta,\gamma)$.
We then have
\begin{align}
  \frac{K_{0n} - \theta \log(n)}{\theta\sqrt{\log(n)}}
  \stackrel{d}{\to}
  \mathcal{N}(0,1)
\end{align}
\end{theorem}


\begin{proof}
We prove this theorem by showing
\begin{align}
\E{ \exp\left(\frac{is K_{0n}}{\sqrt{\log(n)}} \right)}
\sim
e^{is \theta \sqrt{\log(n)}} e^{-\theta s^2/2},
\label{eq_tot_number_cycles_limit}
\end{align}
and then applying L\'{e}vy's continuity theorem.
Since $O(.)$ in \eqref{eq_mod_poisson_class_F_r_theta_plus_error} resp. \eqref{eq_mod_poisson_class_eF_r_theta_gamma_plus_error}  is uniform in $w$ on compact sets, we can chose $w = e^{is/\sqrt{\log(n)}}$.
We get
\begin{align*}
\ET{\exp\left( \frac{is}{\sqrt{\log(n)}} K_{0n}  \right)}
 &\sim
n^{\theta \left ( e^{\frac{is}{\sqrt{\log(n)}}}  -1 \right)}
=
\exp\left( \theta \log(n)\bigl( e^{\frac{is}{\sqrt{\log(n)}}} -1 \bigr) \right)\\
&=
\exp\left( \theta \log(n)\left( \frac{is}{ \sqrt{\log(n)} } -\frac{s^2}{2 \log(n)} + O\left(\log^{-3/2}(n)\right)   \right) \right)\\
&\sim
e^{is \theta \sqrt{\log(n)}} e^{-\theta s^2/2}.
\end{align*}
\end{proof}

In fact $K_{0n}$ converges in a stronger sense, namely in the mod-Poisson sense:
\begin{definition}
\label{def_mod_poisson_convergence}
We say that a sequence of random variables $Z_n$ converges
in the strong mod-Poisson sense with parameters $\la_n$ if
\begin{align}
  \lim_{n\to\infty}
  \exp\bigl(\la_n (1- e^{is}) \bigr)  \E{e^{is Z_n}}
  =
  \Psi(s)
\end{align}
locally uniform for each $s\in \R$ and $\Psi(s)$ a continuous function with $\Psi(0)=1$.
\end{definition}
The mod-Poisson convergence is stronger than the normal convergence in Theorem~\ref{thm_limit_theorem_total_cycles} since mod-Poisson convergence implies normal convergence (see \cite[Proposition~2.4]{Ashkan_mod_poisson}. Mod-Gaussian convergence and mod-Poisson convergence were first introduced in \cite{Ashkan_mod_poisson} and \cite{JKN}.  Details on mod-Poisson and mod-Gaussian convergence and its use in number theory, probability theory and random matrix theory can be found in \cite{Ashkan_mod_poisson} and \cite{JKN}.

\begin{theorem}
\label{thm_mod_poisson_convergence}
Let $g_\Theta(t)$ be of class $\mathcal{F}(r,\theta)$ or of class $e\mathcal{F}(r,\theta,\gamma)$.
Then the sequence $K_{0n}$ converges in the strong mod-Poisson sense with parameters $K+\theta \log(n)$  with limiting function $\frac{\Gamma(\theta)}{\Gamma(\theta e^{is})}$.
\end{theorem}
\begin{proof}
This follows immediately from Lemma~\ref{lem_generating_K0n}.
\end{proof}
\textbf{Remark:}

In fact we could show mod-Poisson convergence with parameter $\theta \log n$ instead of $K+\theta \log n$; then the limiting function would be slightly modified.

It is natural in this context to approximate $K_{0,n}$ with a Poisson distribution with parameter $\la_n =K+ \theta \log(n)$ (or simply $\la_n = \theta \log(n)$).
To measure the distance between $K_{0n}$ and a Poisson distribution, we introduce some distances between measures.
\begin{definition}
Let $X$ and $Y$ be  integer valued random variables with distributions $\mu$ and $\nu$. We then define
\begin{enumerate}
 \item the point metric
\begin{align}
d_{loc}(X,Y)
:=
d_{loc}(\mu,\nu)
:=
\sup_{j\in\Z} | \mu\set{j} - \nu\set{j}|
\end{align}
\item the Kolmogorov distance
\begin{align}
d_{K}(X,Y)
:=
d_{K}(\mu,\nu)
:=
\sup_{j\in\Z} | \mu\set{(-\infty,j]} - \nu\set{(-\infty,j]}|
\end{align}
%
\end{enumerate}
\end{definition}

\begin{lemma}
Let $P_{K+\theta \log(n)}$ be a Poisson distributed random variable with parameter $K+\theta \log(n)$.
If $g_\Theta$ is of class $\mathcal{F}(r,\theta)$ then
\begin{align}
 d_{loc}(K_{0n},P_{K+\theta \log(n)})
 &\leq
 \frac{c_1}{\log(n)}
\label{eq_d_loc_Kon}\\
 d_{K}(K_{0n},P_{K+\theta \log(n)})
 &\leq
 \frac{c_2}{\sqrt{\log(n)}}
\label{eq_d_K_Kon}
\end{align}
with $c_1>0, c_2>0$ independent of $n$.
If $g_\Theta$ is of class $e\mathcal{F}(r,\theta,\gamma)$, then
\begin{align}
 d_{loc}(K_{0n},P_{K+\theta \log(n)})
 \leq
 \frac{c_3}{\log(n)}
\label{eq_d_loc_2_Kon}
\end{align}
\end{lemma}

\begin{proof}
Formulas \eqref{eq_d_loc_Kon} and \eqref{eq_d_K_Kon} can be proven with Proposition~3.1 and Corollary~3.2 in \cite{mod_discrete}
for $\chi = e^{(K+\theta \log(n)) (e^{is} -1)}$ , $\psi_\nu(s) = \frac{\Gamma(\theta)}{\Gamma(\theta e^{is})}$,
$\psi_\mu(s) = 1$ and $\eps$ the error term in Lemma~\ref{lem_generating_K0n}.\\
We cannot prove \eqref{eq_d_loc_2_Kon} with Proposition~3.1 and Corollary~3.2 in \cite{mod_discrete}, since the characteristic function of $K_{0n}$ does not fulfill the requested conditions. But we can modify the method in \cite{mod_discrete}.  We have by the Fourier inversion formula
\begin{align}
 \mu\set{j} - \nu\set{j}
 =
\frac{1}{2\pi}
\int_{-\pi}^{\pi} e^{ij s} \bigl(\phi_\mu(s) - \phi_\nu(s) \bigr) \ ds
\end{align}
with $\phi_\mu(s), \phi_\nu(s)$ the characteristic functions of $\mu$ and $\nu$.
The characteristic function of $K_{0n}$ is given by \eqref{eq_mod_poisson_class_eF_r_theta_gamma_plus_error} and the characteristic function of $P_{K+\theta \log(n)}$ is $\exp\bigl((K+\theta \log(n) )( e^{is} -1)\bigr)$. We thus get
\begin{align*}
 &|\Pb{K_{0n} =j} - \Pb{P_{\theta \log(n)} =j}|\\
&=
\frac{1}{2\pi} \left|
\int_{-\pi}^{\pi} e^{ij s}
e^{(K+\theta \log(n)) \bigl( e^{is} -1\bigr)} \left( \frac{\Gamma(\theta)}{\Gamma(\theta e^{is})} -1 \right) +  O\left( \frac{\log(n)}{n^\gamma} \right) \ ds \right|\\
&\leq
\frac{1}{2\pi}
\int_{-\pi}^{\pi}
e^{-(K+\theta \log(n) )s^2/2} \left| \frac{\Gamma(\theta)}{\Gamma(\theta e^{is})} -1 \right| \ ds + O\left( \frac{\log(n)}{n^\gamma} \right)\\
&\leq
\frac{1}{2\pi}
\int_{-\pi}^{\pi}
e^{-(K+\theta \log(n)) s^2/2} \gamma |s| \ ds + O\left( \frac{\log(n)}{n^\gamma} \right)
\end{align*}
We are now in the same situation as in the proof of Proposition~2.1 in \cite{mod_discrete}.
One  can thus use exactly the same arguments to get the desired upper bound. We thus omit the details.
\end{proof}

We now wish to deduce some large deviations estimates from the mod-Poisson convergence. For this we use results from some work in progress \cite{NW} which establishes links between mod-* convergence (e.g. mod-Poisson or mod-Gaussian convergence) and precise large deviations. More precisely the framework is as follows: we assume we are given a sequence of random variables $X_n$ such that $\varphi_n(z)= \mathbb{E}[e^{zX_n}]$ exists in a strip $-\varepsilon<Re(z)<c$, with $c$ and $\varepsilon$ positive numbers. We assume that there exists an infinitely divisible distribution with moment generating function $\exp(\eta(z))$ and an analytic function $\phi(z)$ such that locally uniformly in $z$
\begin{equation}\label{raterate}
\exp \left( -t_n \eta(z)\right)\varphi_n(z)\to \phi(z),\;n\to\infty ,
\end{equation}
for $-\varepsilon<Re(z)<c$ and some sequence $t_n$  tending to infinity. We further assume that $\phi(z)$ does not vanish on the real part of the domain.

\begin{theorem}[\cite{NW}]
Let $(X_n)$ be a sequence of random variables which satisfies the assumptions above. Assume further that these and the corresponding infinitely divisible distribution have minimal lattice $\mathbb{N}$. Assume further that the rate of convergence in (\ref{raterate}) is faster than any power of $1/t_n$. Let $x$ be a real number such that $t_n x\in\N$ and such that there exists $h$ with $\eta'(h)=x$. Noting $I(z)=hx-\eta(h)$  the Fenchel-Legendre transform, we have the following asymptotic expansion
\begin{equation}\label{asyexx}
\mathbb{P}[X_n=x t_n]\sim \dfrac{\exp\left(-t_n I(x) \right)}{\sqrt{2 \pi t_n \eta^{''}(h)}} \left( \phi(h)+\dfrac{a_1}{t_n}+\dfrac{a_2}{t_n^2}+\ldots \right), n\to\infty
\end{equation}
where (\ref{asyexx}) has to be understood in the sense that for every integer $N$, one has
$$\mathbb{P}[X_n=x t_n]\sim \dfrac{\exp\left(-t_n I(x) \right)}{\sqrt{2 \pi t_n \eta^{''}(h)}} \left( \phi(h)+\dfrac{a_1}{t_n}+\ldots+\dfrac{a_{N-1}}{t_n^{N-1}}+O\left(\dfrac{1}{t_n^N}\right) \right).$$
The coefficients $(a_j)$ are real, depend only on $h$ and can be computed explicitly.
\end{theorem}
We now apply the above result to the random variable $X_n=K_{0n}-1$ in order to have a random variable distributed on $\N$. It is easy to see that in this case $\eta(z)=e^z-1$, $h=\log x$ and $I(z)=z\log z-z+1$.  It follows form Theorem \ref{thm_mod_poisson_convergence} and Lemma \ref{lem_generating_K0n} that $K_{0n}-1$ satisfies the assumptions of the above theorem with $\phi(z)=\dfrac{\Gamma(\theta)}{\Gamma(e^z) \Gamma(\theta e^z)}$ and $t_n=K+\theta \log n$. Using Stirling's formula we obtain the following large deviations estimates (the case $K=0$ and $\theta=1$ corresponds to the result of Hwang \cite{Hwang95}).
\begin{theorem}
Let $X_n=K_{0n}-1$. Let $x\in\R$ such that $t_n x\in\N$, where $t_n=K+\theta \log n$. Then
$$\mathbb{P}[X_n=x t_n]=e^{-t_n}\dfrac{t_n^k}{k!}\left(\dfrac{\Gamma(\theta)}{\Gamma(x)\Gamma(\theta x)}+O\left(\dfrac{1}{\log x}\right) \right).$$
\end{theorem}

\textbf{Remark:}

In fact one could obtain an arbitrary long  expansion in the above above theorem.

%% file: section_examples.tex
\section{Some examples}
\label{sec_examples}

In this section we consider some examples of sequences $\Theta$ and check if $g_\Theta(t)$ is of class $\mathcal{F}(r,\theta)$ or of class $e\mathcal{F}(r,\theta,\gamma)$.

\subsection{Simple sequences}

\subsubsection{The Ewens measure}

The simplest possible sequence is the constant sequence $\theta_m = \theta$.
This case is known as Ewens measure and is well studied, see for example \cite{barbour}.
We have
\begin{align}
g_\Theta(t) = \theta \log \left( \frac{1}{1-t} \right).
\end{align}
We thus have $g_\Theta(t)\in \mathcal{F}(1,\theta)$ and our results apply.

\subsubsection{The condition $\sum_{m=1}^\infty \frac{|\theta_m-\theta|}{m} < \infty$}
\label{sec_sum_theta_m-theta/m}
We have for $t \to 1$ and $|t|\leq 1$
\begin{align}
  g_\Theta(t)
  &=
 \sum_{m=1}^\infty \frac{\theta}{m} t^m+ \sum_{m=1}^\infty \frac{\theta_m-\theta}{m} t^m
  =
  \theta \log\left(  \frac{1}{1-t}   \right) + \sum_{m=1}^\infty \frac{\theta_m-\theta}{m}  t^m \nonumber\\
  &=
  \theta \log\left(  \frac{1}{1-t}   \right)
    +
  \sum_{m=1}^\infty \frac{\theta_m-\theta}{m}
    +
  \sum_{m=1}^\infty \frac{\theta_m-\theta}{m}  (t^m-1) \nonumber\\
&=
  \theta \log\left(  \frac{1}{1-t}   \right)
    +
  \sum_{m=1}^\infty \frac{\theta_m-\theta}{m}
    +
  o(1).
\label{eq_sum_theta_m-theta_/m}
\end{align}

It is clear that $g_\Theta(t) \in e\mathcal{F}(1,\theta,\gamma)$ if we assume $|\theta_m-\theta| = O(m^{-\gamma})$.
It is not clear from \eqref{eq_sum_theta_m-theta_/m} if $g_\Theta(t) \in \mathcal{F}(1,\theta)$ or not, even when one can extend $g_\Theta(t)$ holomorphically to some $\Delta(1,R,\phi)$.
If a holomorphic extension is available, then one can modify Theorem~\ref{thm_hwang} by replacing $O(t-1)$ by $o(1)$ in the definition of $\mathcal{F}(1,\theta)$. The only difference in the result is that one has to replace the error term $O\left(\frac{1}{n} \right)$ in \eqref{eq_thm_hwang} by $o(1)$. We do not prove this here since one only has to do some minor changes in the proof of Theorem~\ref{thm_hwang}.

\subsubsection{The condition $\sum_{m=1}^\infty |\theta_m-\theta| < \infty$}
\label{sec_sum_theta_m-theta}
We have for $t \to 1$ and $|t|\leq 1$
\begin{align}
  g_\Theta(t)
  &=
  \theta \log\left(  \frac{1}{1-t}   \right)
    +
  \sum_{m=1}^\infty \frac{\theta_m-\theta}{m}
    +
  \sum_{m=1}^\infty \frac{\theta_m-\theta}{m}  (t^m-1) \nonumber\\
  &=
  \theta \log\left(  \frac{1}{1-t}   \right)
    +
  \sum_{m=1}^\infty \frac{\theta_m-\theta}{m}
    +
  (t-1) \sum_{m=1}^\infty \frac{\theta_m-\theta}{m} \left(\sum_{k=0}^{m-1} t^k \right) \nonumber\\
&=
  \theta \log\left(  \frac{1}{1-t}   \right)
    +
  \sum_{m=1}^\infty \frac{\theta_m-\theta}{m}
  +
 O(t-1)
\label{eq_sum_theta_m-theta}
\end{align}
As before $g_\Theta(t) \in e\mathcal{F}(1,\theta,\gamma)$ if we assume $|\theta_m-\theta| = O(m^{-\gamma})$.
If $g_\Theta(t)$ can be  holomorphically extended to some $\Delta(1,R,\phi)$, then \eqref{eq_sum_theta_m-theta} shows that $g_\Theta(t) \in \mathcal{F}(1,\theta)$ (if we have an asymptotic expansion near $1$ in $\Delta(1,R,\phi)$).

\subsubsection{The sequence $\theta_m = e^{-\alpha_m}$ and $\sum_{m=1}^\infty \frac{|\alpha_m-\alpha|}{m} < \infty$ or $\sum_{m=1}^\infty |\alpha_m-\alpha| < \infty$}

Both conditions on $\alpha_m$ ensure that $|e^{-\alpha_m}-e^{-\alpha}| \leq C |\alpha_m-\alpha|$.
One can use the same argumentation as above to see that  in the first  case
\begin{align}
g_\Theta(t)
&=
  e^{-\alpha} \log\left(  \frac{1}{1-t}   \right)
    +
  \sum_{m=1}^\infty \frac{e^{-\alpha_m}-e^{-\alpha} }{m}
  +
 o(1)
\label{eq_sum_alpha_m-alpha/m}
\end{align}
and in the second case
\begin{align}
g_\Theta(t)
&=
  e^{-\alpha} \log\left(  \frac{1}{1-t}   \right)
    +
  \sum_{m=1}^\infty \frac{e^{-\alpha_m}-e^{-\alpha} }{m}
  +
 O(t-1).
\label{eq_sum_alpha_m-alpha}
\end{align}
This shows that we are in the same situation as in Section~\ref{sec_sum_theta_m-theta/m} and Section~\ref{sec_sum_theta_m-theta}.

\subsection{polylogarithm}
\label{sec_poly_logarithm}

Let $\theta_m = m^\delta$ with $\delta \neq 0$. We then have
\begin{align}
 \mathrm{Li}_{1+\delta}
 :=
 g_\Theta(t)
 =
 \sum_{m=1}^\infty \frac{1}{m^{1+\delta}} t^m.
\end{align}
The functions $\mathrm{Li}_{1+\delta}$ are known as the polylogarithm.
A simple computation shows that the convergence radius of $\mathrm{Li}_{1+\delta}$ is $1$.
It was shown by Ford in \cite{Ford} with Theorem~\ref{thm_lindelof_integral_representation} that the polylogarithm can extend holomorphically to the whole complex plane split along the axis $\R_{\geq 1}$.
The asymptotic behaviour of $\mathrm{Li}_{1+\delta}$ near $1$ can be computed with the Mellin transform.
This has been done in \cite[Section~3]{MR1678788}. We just state  here the result.\\
The case $\delta=0$ is trivial since we then have $\mathrm{Li}_{1} = -\log \left(\frac{1}{1-t} \right)$.\\
If $\delta\in\set{1,2,3,\cdots}$, then
\begin{align}
  \mathrm{Li}_{1+\delta}
  =
  \zeta(\delta+1)
  +
  O(t-1)
\label{eq_expansion_polylog_integer}
\end{align}
and for  $\delta\notin\set{0,1,2,\cdots}$
\begin{align}
  \mathrm{Li}_{1+\delta}
  =
  \Gamma(-\delta)(1-t)^{\delta}
  +
  \zeta(\delta+1)
  +
  O(t-1)
\label{eq_expansion_polylog_non_integer}
\end{align}
For $\delta >0$ we can apply Theorem~\ref{thm_hwang}, but the main term in the asymptotic expansion is $0$. 
%

%
If $\delta <0$ then $\mathrm{Li}_{1+\delta}$ is not of class $\mathcal{F}(r,\theta)$ nor of class $e\mathcal{F}(r,\theta,\gamma)$.
This shows that we cannot apply Theorem~\ref{thm_hwang} nor Theorem~\ref{thm_hwang_2}.
Ercolani and Ueltschi have shown in \cite{cycle_structure_ueltschi} that $K_{0n}$ converge for $\delta<0$ in distribution to a shifted  Poisson distribution (without re-normalization) and for $\delta>0$ they have shown $\E{K_{0n}} \approx A n^{\frac{\delta}{1+\delta}}$.
It is not yet known if $K_{0n}$ converge in distribution (after re-normalization). The method of Ercolani and Ueltschi bases also on generating functions, but they use the saddle point method to extract the asymptotic behaviour of $\nth{g}$. In this way they can have weaker assumptions on the sequence $\Theta$, but do not get information on the error term.

\subsection{ $\theta_m = \exp(c m^\theta)$}
\label{sec_exponential_growth}

A simple computation shows that the convergence radius $r$ of
$g_\Theta(t)$ is
\begin{align}
 r = \begin{cases}
      1      & \text{ if } \theta < 1,\\
      e^{-c} & \text{ if } \theta = 1,\\
      0      & \text{ if } \theta > 1 \text{ and } c>0,\\
      \infty      & \text{ if } \theta > 1 \text{ and } c<0
     \end{cases}
\end{align}
We can not apply our method for $\theta>1$ and for $\theta=1$ we have $g_\Theta(t) = -\log(1-te^c)$.
We thus restrict ourselves to $\theta <1$.
\begin{lemma}
For $\theta<1$, the function $g_\Theta(t)$ can  holomorphically extended to the whole complex plane split along the axis $\R_{\geq 1}$.
\end{lemma}
\begin{proof}
As simple computation shows that $\left|\frac{\exp(c z^\theta)}{z} \right| \leq 2\exp(|c| |z|^\theta)$ for $\Re(z) \geq \frac{1}{2}$.
Since $\theta<1$, we can find for each $A \in ]0,\pi[$ a $C>0$ with $\left|\frac{\exp(c z^\theta)}{z} \right| \leq C \exp(A|z|)$.
This shows that we can apply Theorem~\ref{thm_lindelof_integral_representation} with $\phi = \frac{\exp(c z^\theta)}{z}$.
\end{proof}
We now determine the asymptotic behaviour near $1$.\\
\underline{Case $\theta<0$}\\
We have
\begin{align}
 g_\Theta(t)
 =
 \sum_{m=1}^\infty \exp(c m^\theta) \frac{t^m}{m}
 =
 \sum_{k=0}^\infty \sum_{m=1}^\infty \frac{ c^k}{k!} m^{k\theta-1} t^m
 =
 \sum_{k=0}^\infty \frac{ c^k}{k!} \mathrm{Li}_{1-k\theta}(t).
\label{eq_super_exp_theta<0_asymp}
\end{align}
We know the expansion of each summand near $1$, but we have to justify that we plug them in.
We first look at the case $|t| \leq 1$. We use the same argumentation as in \eqref{eq_expansion_polylog_integer} to see that
\begin{align}
 \sum_{m=1}^\infty \frac{1}{m^{1-k\theta}} t^m
 &=
 \sum_{m=1}^\infty \frac{1}{m^{1-k\theta}}   + (t-1)\sum_{m=1}^\infty \frac{1}{m^{1-k\theta}} \left( \sum_{k=1}^m t^k   \right)\nonumber\\
 &=
\sum_{m=1}^\infty \frac{1}{m^{1-k\theta}}   + O\left( (t-1) \sum_{m=1}^\infty \frac{1}{m^{2}}  \right)
\end{align}
with $O(.)$ uniform for $-k\theta \geq 2$.
We thus can this put into \eqref{eq_super_exp_theta<0_asymp} and get for $|t|<1$
\begin{align}
 g_\Theta(t)
 =
\log \left( \frac{1}{1-t}  \right) + K + O(t-1) \text{ for } t\to 1, |t|<1
\end{align}
One  can now use the same argumentation as in \cite[Section~3]{MR1678788} with the Mellin transformation to see that the expansion is also valid for $t$ in some $\Delta_0$.\\


\underline{Case $c<0, 0<\theta<1$}\\
In this case it is easy to see that $g_\Theta(t) = K + O(1-t)$ for $t\to 1$. 
We thus cannot apply Theorem~\ref{thm_hwang}  nor Theorem~\ref{thm_hwang_2}. This case was also considered by 
Ercolani and Ueltschi. They have shown that $K_{0n}$  converge as in polylogarthm case in distribution to a shifted Poisson distribution.

\underline{Case $c>0, 0<\theta<1$}\\

The behaviour of $K_{0n}$ in this case is an open question.

%% file: section_general_generating.tex
\section{The generalized weighted measure}
\label{sec_more_general_measures}

In this section we introduce the generalized weighted measure. This is defined as
\begin{definition}
   Let $F_m: \N \to \R_{>0}$  be given for $m\geq 1$ with $F_m(0)=1$. We then define
    \begin{align}
      \PF{\sigma}
      :=
      \frac{1}{n! h_n(F)} \prod_{m=1}^n F_m \bigl(C_m(\sigma) \bigr)
    \end{align}
    with $h_n(F)$ a normalization constant.
\end{definition}
It follows immediately from the definition that $\PT{.} = \PF{.}$ with $F_m(k) = \theta_m^k$.
Our approach has been   based so far on generating functions,  and more especially on Lemma~\ref{lem_cycle_index_theorem}.
It is obvious that Lemma~\ref{lem_cycle_index_theorem} cannot be used anymore for general functions $F_m$,
but we can prove a more general version of it:

\begin{lemma}
\label{lem_cycle_index_theorem_general}
 Let $A_m: \N \to \C$ for $m\geq 1$ be given with $A_m(0)=1$ for $m\geq 1$. We then have as formal power series

\begin{align}
 \sla t^{|\la|} \prod_{m=1}^{|\la|} \bigl(A_m(C_m)  \bigr)
=
\prod_{m=1}^\infty \mathrm{EG}\left(A_m,\frac{t^m}{m} \right)
\text{ with }
\mathrm{EG}(A,t):= \sum_{k=0}^\infty \frac{A(k)}{k!} t^k.
\label{eq_general_cycle_index_theorem}
\end{align}
with $C_m = C_m(\la)$  as in Definition~\ref{def_Cm_and_K0n}
%
\end{lemma}
\begin{proof}
 We have
\begin{align}
 \slan t^{n}\prod_{m=1}^n \bigl( A_m(C_m) \bigr)
&=
\sum_{\substack{c_1,\cdots,c_n \in\N \\ \sum_{m=1}^n m c_m=n}}
\frac{\prod_{m=1}^n A_m(c_m)t^{mc_m}}{\prod_{m=1}^n c_m! m^{c_m}} \nonumber\\
&=
\sum_{\substack{c_1,\cdots,c_n \in\N \\ \sum_{m=1}^n m c_m=n}}
\prod_{m=1}^n \frac{A_m(c_m)}{c_m!} \left( \frac{t^m}{m} \right)^{c_m}.
\end{align}
We get

\begin{align}
 \sla t^{|\la|}\prod_{m=1}^{|\la|} A_m(C_m)
&=
\sum_{\substack{(c_m)_{m=1}^\infty,\\ \sum_{m=1}^\infty m c_m < \infty }}
\prod_{m=1}^{|\la|} \frac{A_m(c_m)}{c_m!} \left( \frac{t^m}{m} \right)^{c_m}
=
\prod_{m=1}^\infty \left( \sum_{k=0}^\infty \frac{A_m(k)}{k!} \left( \frac{t^m}{m} \right)^{k}   \right)\nonumber\\
&=
\prod_{m=1}^\infty \mathrm{EG}\left(A_m,\frac{t^m}{m} \right)
\end{align}
by definition of $\mathrm{EG}(F,t)$.
\end{proof}
As before we can write down the generating functions for the cycle numbers and for $K_{0n}$:

\begin{theorem}
\label{thm_generating_allg}
We have
\begin{align}
   \sum_{n=0}^\infty t^n h_n(F)
   =
   \prod_{m=1}^\infty EG \left(F_m, \frac{t^m}{m} \right).
\label{eq_generating_hn_allg}
 \end{align}
If $b\in\N$ and $s_1,\cdots,s_b \in\R$ are fixed, then
 \begin{align}
   \sum_{n=0}^\infty t^n h_n(F) \E{ \exp\left( i \sum_{m=1}^b s_m C^{(n)}_m \right)  }
   =
   \prod_{m=1}^b EG \left(F_m, e^{is_m} \frac{t^m}{m} \right)
   \prod_{m=1}^b EG \left(F_m, \frac{t^m}{m} \right)
\label{eq_generating_cycles_allg}
 \end{align}
and for each $w\in\C$
 \begin{align}
   \sum_{n=0}^\infty t^n h_n(F) \E{ \exp\left( w K_{0n} \right)  }
   =
   \prod_{m=1}^\infty EG\left(F_m, w \frac{t^m}{m} \right)
\label{eq_generating_cycles_tot_allg}
 \end{align}
\end{theorem}
\begin{proof}
The identity \eqref{eq_generating_hn_allg} follows from \eqref{eq_generating_cycles_allg} by choosing $s_1=\cdots = s_b=0$.
It is thus enough to prove \eqref{eq_generating_cycles_allg}. We have

\begin{align}
  \sum_{n=0}^\infty
   h_n(F) \EF{\exp \left(i\sum_{m=1}^b s_m C_m^{(n)} \right)} t^n
   &=
  \sla  t^{|\la|} \exp \left(i\sum_{m=1}^b s_m C_m^{(n)} \right)
  \prod_{m=1}^n F_m (C_m)\nonumber\\
  &=
  \sla  t^{|\la|} \left(\prod_{m=1}^b  (e^{is_m})^{ C_m^{(n)}} F_m(C_m^{(n)})\right)
  \left(\prod_{m=b+1}^{l(\la)} F_m(C_m^{(n)}) \right)
\end{align}
We now use Lemma~\ref{lem_cycle_index_theorem_general} with
\begin{align}
  A_m(k) = \left\{
             \begin{array}{ll}
                (e^{is_m})^k F_m(k), & \hbox{if } 1 \leq m \leq b ; \\
               F_m(k), & \hbox{if } m> b.
             \end{array}
           \right.
\end{align}
A simple computation then shows that
\begin{align}
EG\left(A_m, \frac{t^m}{m} \right)
=
EG\left(F_m, e^{is_m} \frac{t^m}{m} \right)
\end{align}
for $1 \leq m \leq b$.
This proves \eqref{eq_generating_cycles_allg}. The proof of  \eqref{eq_generating_cycles_tot_allg} is similar.
\end{proof}

We  have thus found the generating functions in this general setting. To get the asymptotic behaviour as in Section~\ref{sec_central_limit_tot_number_cycles} and Section~\ref{sec_central_limit_number_cycles}, we need some analyticity assumptions. We now give here some examples

%
%
%
%
%
%
%
%
%
%
%
%
%
%
%
\subsection{exp-polynomial weights}
\label{sec_apperance of polylogartihms}

Suppose that a polynomial $P(t) = \theta t+ \sum_{k=2}^d b_k t^k$ is given with $b_k \geq 0$ and $\theta>0$. We then define $F_m$ implicit by the equation
\begin{align}
 EG(F_m,t) = \exp( P(t)).
\label{eq_def_F_polylogarithm}
\end{align}
It is easy to see from \eqref{eq_def_F_polylogarithm} that $F_m(0) =1$ and $F_m(c) \geq 0$.
The functions $F_m$ thus generate a probability measure on $S_n$. We do not need an explicit expression for $F_m$ since we are only interested in the asymptotic behaviour of $C_m$ and $K_{0n}$.
We have for $|t|<1$:
\begin{align}
   \sum_{n=0}^\infty t^n h_n
   &=
   \prod_{m=1}^\infty \exp\left(P\left( \frac{t^m}{m} \right)\right)
   =
   \exp\left(\theta \left(\sum_{m=1}^\infty \frac{t^m}{m} \right)
   +
   \sum_{k=2}^d b_k \sum_{m=1}^\infty \left(\frac{t^m}{m} \right)^k \right)\nonumber\\
   &=
   \exp\left( -\theta \log(1-t) + \sum_{k=2}^d b_k \mathrm{Li}_k(t^k) \right)
\end{align}
with $\mathrm{Li}_k(t) :=  \sum_{m=1}^\infty \frac{t^m}{m^k}$.
%
%
We get that $-\theta \log(1-t) + \sum_{m=2}^d \mathrm{Li}_k(t^k)$ is of class $\mathcal{F}(\theta,1)$ and we can apply
Theorem~\ref{thm_hwang}. It is now obvious that we get in this situation the same asymptotic behaviour as for the Ewens measure.
\subsection{spatial random permutations}
\label{sec_spatial_random_permutations}

The measure in this section comes from physics and arises from a model for the Bose gas and has a connection to
Bose-Einstein condensation. We give here only the definition and a very brief idea of the model behind, but avoiding further details. We refer to \cite{spatial_random_permutations} for a more comprehensive overview.

We follow here the physicists notation and write $\theta_m = e^{-\alpha_m}$.
Let $\xi:\R^d \to \R$ be a continuous function and $\Lambda$ be a lattice in $\R^d$ be given.
Assume that for all $k\in\Lambda$.
\begin{align}
 0 \leq  e^{-\eps(k)}:= \int_{\R^d} e^{-\xi(x)} e^{-2\pi i <k,x>}\ dx
\ \text{ and }\
\sum_{k\in \Lambda } e^{-\eps(k) } < \infty.
\end{align}
We then define
\begin{align}
F_m(c)
=
F_m^{(\alpha,\xi)}(c)
=
\left(e^{-\alpha_m} \sum_{k\in \Lambda } e^{-\eps(k) m} \right)^c.
\end{align}
The functions $F_m^{(\alpha,\xi)}$ are well defined since there exists only finitely many $k\in \Lambda$ with $e^{-\eps(k) m} > 1$.
As before, we use $h_n=h_n^{(\alpha,\xi)}$ for the normalization constant.

We now describe the physical model behind this. Let $D$ be a fundamental domain of $\R^d /\Lambda$ and $x_1,\cdots,x_n$ be $n$ particles in $D$. The function $\xi$ plays the role to penalize certain configurations of the $x_i$. More precise, the probability of a given configuration is defined as
\begin{align}
\mathbb{P}_{n,dx}\left[\sigma, dx \right]
=
\frac{1}{h_n n!} \exp\left(-\sum_{m=1}^n \xi\bigl( x_m -x_{\sigma(m)}   \bigr) - \sum_{m=1}^n C_m \alpha_m \right).
\label{eq_def_spatial_phys}
\end{align} 
where $dx$ is the normalized Lebesgue measure on $D$. 
The probability measure $\mathbb{P}_{n,dx} [.]$ now induces a probability measure $\mathbb{P}_{n}[.]$ on $S_n$ (by averaging over $D$).
It is not obvious, but it can be proven that $\mathbb{P}_{n}[.]$ and $\PF{.}$ are the same measures, see \cite[Proposition~3.1]{spatial_random_permutations}.\\

We now compute the asymptotic behaviour of $C_m$ and $K_{0n}$ with respect to this measure.
The lattice in \cite{spatial_random_permutations} depends on $n$ and is so chosen that the density $\rho := \frac{|D|}{n}$ is fix. We assume here that the lattice $\Lambda$ is independent of $n$.
The reason why we do this that we need in our approach that the weights are independent of $n$.
We define
\begin{align}
  g^{(\alpha)}(t)
  &:=
  \sum_{m=1}^\infty \frac{e^{-\alpha_m}}{m} t^m
  \label{eq_def_gen_hn_with_alpha}\\
  g^{(\alpha,\xi)}(t)
  &:=
  \sum_{k\in \Lambda^d}  g^{(\alpha)}(e^{-\eps(k)}t)
  =
  \sum_{k\in \Lambda} \sum_{m=1}^\infty  \frac{e^{-\alpha_m}}{m} t^m  (e^{-\eps(k)})^m
\label{eq_def_gen_hn_with_alpha_xi}
\end{align}
The functions $g^{(\alpha)}(t)$ and $g^{(\alpha,\xi)}(t)$ are both formal power series in $t$ since by assumption $e^{-\eps(k)} \geq 0$ and the coefficient of each $t^m$ is finite.

The function $g^{(\alpha)}(t)$ agrees with $g_\Theta(t)$ for $\Theta = \left( e^{-\alpha_m}\right)_{m\in\N}$.
The difference to the Section~\ref{sec_central_limit_number_cycles} and~\ref{sec_central_limit_tot_number_cycles} is that
$g^{(\alpha,\xi)}(t)$ will play role of $g_\Theta(t)$ and not $g^{(\alpha)}(t)$, i.e. the asymptotic behaviour of $C_m$ and $K_{0m}$ depends directly on the analytic properties of $g^{(\alpha,\xi)}$ (see below).\\
We use the same argumentation as in the previous sections to get the asymptotic behaviour of $C_m$ and $K_{0n}$.
We begin with the generating functions:
\begin{theorem}
\label{thm_gen_for_alpha_xi}
Assume that $\sum_{k\in \Lambda} e^{-\eps(k) }$ is convergent.
We then have the following identities (as formal power series)
\begin{align}
G^{(\alpha,\xi)} (t)
:=
\sum_{n=0}^\infty t^n h_n^{(\alpha,\xi)}
&=
\exp \left(  g^{(\alpha,\xi)}(t) \right)
\label{eq_hn_gen_ueltschi}\\
\sum_{n=0}^\infty t^n h_n^{(\alpha,\xi)} \E{e^{\left(i\sum_{m=1}^b s_m C_m^{(n)}\right)}}
 &=
 G^{(\alpha,\xi)} (t)
 \prod_{m=1}^b
 \exp \left((e^{is}-1) \frac{t^m}{m}  \left(e^{-\alpha_m} \sum_{k\in \Lambda} e^{-\eps(k) m} \right)\right)
\label{eq_Cm_gen_ueltschi}\\
\sum_{n=0}^\infty h_n^{(\alpha,\xi)} \E{\exp\bigl( w K_{0n}  \bigr)} t^n
 &=
\exp \left(w   g^{(\alpha,\xi)}(t) \right).
\label{eq_K0n_gen_ueltschi}
\end{align}
\end{theorem}
\begin{proof}
We begin with \eqref{eq_hn_gen_ueltschi}. We have
\begin{align}
\mathrm{EG}\left(F_m,\frac{t^m}{m}\right)
&=
\sum_{c=0}^\infty \frac{F_m(c)}{c!} \left( \frac{t^m}{m}\right)^c
=
\sum_{c=0}^\infty \frac{1} {c!} \left( \frac{t^m}{m} e^{-\alpha_m} \sum_{k\in \Lambda } e^{-\eps(k) m} \right)^c\\
&=
\exp \left(\frac{t^m}{m}  \left(e^{-\alpha_m} \sum_{k\in \Lambda} e^{-\eps(k) m} \right)       \right).\nonumber
\end{align}
Thus
\begin{align}
\prod_{m=1}^\infty \mathrm{EG}\left(F_m,\frac{t^m}{m} \right)
&=
\exp \left(\sum_{m=1}^\infty \frac{t^m}{m}  \left(e^{-\alpha_m} \sum_{k\in \Lambda} e^{-\eps(k) m} \right)       \right)
\label{eq_gen_hn_allg_comp}\\
&=
\exp \left(\sum_{ k\in \Lambda}  \sum_{m=1}^\infty e^{-\alpha_m} \frac{(e^{-\eps(k)}t)^m}{m}        \right)
=
\exp \left(\sum_{k\in \Lambda}  g^{\alpha}(e^{-\eps(k)}t) \right)\nonumber
\end{align}
The above expressions are formal power series since $g^{(\alpha,\xi)}(t)$ is a formal power series.
This proves \eqref{eq_hn_gen_ueltschi}.
We have
\begin{align}
 \mathrm{EG}\left(F_m,w\frac{t^m}{m}\right)
=
\exp \left(w \frac{t^m}{m}  \left(e^{-\alpha_m} \sum_{k\in \Lambda} e^{-\eps(k) m} \right)       \right)
\end{align}
The identities \eqref{eq_Cm_gen_ueltschi} and \eqref{eq_K0n_gen_ueltschi} thus follow immediately from Theorem~\ref{thm_generating_allg}.
\end{proof}

The next step is to obtain the asymptotic behaviour of $C_m$ and $K_{0n}$.
For this we have to assume some analytic properties.
In this context, it is natural to assume that $g^{(\alpha)}$ is of class $\mathcal{F}(\alpha,r)$.
\begin{lemma}
\label{lem_g_alpha_xi_analytic}
Let $g^{(\alpha)}$ be of class $\mathcal{F}(\alpha,r)$.
We define
\begin{align}
\widetilde{r}
:=
\min\set{e^{\eps(k)}, k\in\Lambda}
\text{ and }
A:=\#\set{k\in \Lambda, e^{\eps(k)} = \widetilde{r} }.
\end{align}
Then $g^{(\alpha,\xi)} (t)$ is of class  $\mathcal{F}(A\alpha,\widetilde{r} r)$.
\end{lemma}
\begin{proof}
We have by assumption that $g^{(\alpha)}$ is holomorphic in $\Delta_0(r,R,\phi)$ for some $\phi,R$.
Thus $g^{(\alpha)}(e^{-\eps(k)}t)$ is holomorphic in $e^{\eps(k)} \Delta_0(r,R,\phi)$.
This shows that $g^{(\alpha,\xi)}(t)$ is a sum of functions holomorphic in $\widetilde{r} \Delta_0(r,R,\phi)$.
We thus have to show that the sum is convergent.
We have $|g^{(\alpha)}(t)| < \widehat{C}_1 |t|$ for $|t|<r/2$ since by assumption $g^{(\alpha)}(0)=0$.
Let $t$ in $\widetilde{r} \Delta_0(r,R,\phi)$ be fixed.
Since the $\sum_{k\in \Lambda } e^{-\eps(k) }$ is convergent, there are only finitely many $k\in \Lambda^d$ with $|t|e^{-\eps(k)} > \frac{1}{2} r$.
We thus get
\begin{align}
|g^{(\alpha,\xi)} (t)|
=
\left|
\sum_{k\in \Lambda}  g^{(\alpha)}(e^{-\eps(k)}t)
\right|
\leq
\widehat{C}_2  + \sum_{k\in \Lambda^d} \widehat{C}_1 |e^{-\eps(k)}t|
\leq
\widehat{C}_2 +  \widehat{C}_1 |t| \sum_{k\in \Lambda}  |e^{-\eps(k)}| < \infty.
\end{align}
This proves that $g^{(\alpha,\xi)} (t)$ is holomorphic in $\widetilde{r}\Delta_0(r,R,\phi)$.
The other statements are clear.
\end{proof}
We now apply Theorem~\ref{thm_hwang} to obtain:
\begin{theorem}
\label{thm_limit_theorem_cycles_alpha_xi}
Let $g^{(\alpha)}$ be of class $\mathcal{F}(\alpha,r)$ and $\widetilde{r},A$ be as in Lemma~\ref{lem_g_alpha_xi_analytic}.
We then have for each $b\in\N$
\begin{align}
  \left(C_1^{(n)}, C_2^{(n)},  \cdots C_b^{(n)} \right)
  \stackrel{d}{\to}
  \left(Y_1,\cdots,Y_b\right)
\end{align}
with $Y_1,\cdots,Y_b$ independent Poisson distributed random variables with
\begin{align}
\E{Y_m} = \frac{(r\widetilde{r})^m}{m}  \left(e^{-\alpha_m} \sum_{k\in \Lambda} e^{-\eps(k) m} \right)
\end{align}
\end{theorem}
\begin{proof}
The proof of this theorem is the same as the proof of Theorem~\ref{thm_cycle_numbers_generating}
One only has to use Theorem~\ref{thm_gen_for_alpha_xi} and Lemma~\ref{lem_g_alpha_xi_analytic}.
We therefore omit the details.
\end{proof}
\begin{theorem}
\label{thm_limit_theorem_tot_cycles_alpha_xi}
Let $g^{(\alpha)}$ be of class $\mathcal{F}(\alpha,r)$ and $\widetilde{r},A$ be as in Lemma~\ref{lem_g_alpha_xi_analytic}.
Then $K_{0n}$ converges in the strong mod-Poisson sense with parameters $A\alpha \log(n)$  with limiting function $\frac{\Gamma(\alpha)}{\Gamma(\alpha e^{is})}$.

\end{theorem} 